\documentclass[11pt,reqno]{amsart}

\usepackage[margin=1.22in]{geometry}
\usepackage{graphicx}
\usepackage{mathrsfs}
\usepackage{amssymb}
\usepackage{amsmath}
\usepackage{mathtools}
\usepackage{tikz}
\usepackage{tikz-cd}
\usepackage{enumitem}
\usepackage{microtype}
\usepackage{lmodern}
\usepackage{csquotes}
\usepackage{nicefrac}
\usepackage{comment}
\usepackage[colorlinks=true, urlcolor=blue, citecolor=blue, linkcolor=blue]{hyperref}

\numberwithin{equation}{section}
\numberwithin{equation}{section}

\theoremstyle{definition}
\newtheorem{theorem}{Theorem}[section]
\newtheorem*{theorem*}{Theorem}
\newtheorem{proposition}[theorem]{Proposition}
\newtheorem{lemma}[theorem]{Lemma}
\newtheorem{corollary}[theorem]{Corollary}
\newtheorem{remark}[theorem]{Remark}
\newtheorem{definition}[theorem]{Definition}
\newtheorem{question}{Question}
\newtheorem{exam}[theorem]{Example}

\newtheorem{theorema}{Theorem}

\newcommand{\R}{\mathbf R}
\newcommand{\Z}{\mathbf Z}
\newcommand{\CP}{\mathbf{CP}}

\DeclareMathOperator{\Scal}{Scal}
\DeclareMathOperator{\Vol}{Vol}
\DeclareMathOperator{\stsys}{stsys}
\DeclareMathOperator{\comass}{comass}
\DeclareMathOperator{\rank}{rank}
\DeclareMathOperator{\ch}{ch}
\DeclareMathOperator{\Td}{Td}
\DeclareMathOperator{\tors}{tors}

\newcommand{\Ah}{\widehat A}
\newcommand{\Ahcw}{\widehat A\text{-}\mathrm{cw}}
\newcommand{\Ahcwline}{\widehat A\text{-}\mathrm{cw}^{\mathrm{line}}}

\newcommand{\norm}[1]{\left\lVert #1\right\rVert}
\newcommand{\ip}[2]{\left\langle #1,#2\right\rangle}

\title{Stable systolic inequalities via mod n covering}
\author{Aditya Kumar}
\address{Department of Mathematics, University of Maryland, 4176 Campus Dr, College
Park, MD 20742, USA}
\date{}

\begin{document}

\begin{abstract}
We introduce a mod $n$ covering based approach to stable systolic inequalities. The idea is to prescribe a cohomology class mod \(n\) which forces the desired cup product or index to be nonzero, and then find a short integral lift of that class. The method is especially effective in rank two as we can compute the covering constant. As a curvature free application, we improve the stable two systolic bound for \(S^2\times S^2\) to \(2\). The same bound holds for every oriented four manifold with \(b_2=2\). Under a positive scalar curvature lower bound, the mod \(n\) covering method combined with a sharp cowaist inequality for line bundles gives stable two systolic bounds. This gives the sharp stable two systolic inequality for odd complex projective spaces and an \(O(m\log m)\) bound for \((S^2)^m\) when scalar curvature is at least \(2m\). For \(S^2\times S^2\) one gets that every metric with scalar curvature at least \(4\) has stable two systole at most \(8\pi\).
\end{abstract}

\maketitle
\tableofcontents

\section{Introduction}

\subsection{Background} Systolic inequalities compare the minimal volume of nontrivial $k$-cycles on a Riemannian manifold with the total volume. For ordinary homological systoles, such inequalities are often false for \(k\)-cycles with \(k >1\). This phenomenon is known as systolic freedom: in many settings one can make the total volume small while keeping the $k$-volume all nontrivial $k$-cycles large \cite{gromov-systolic,babenko-katz,Katz-suciu,freedman}. Stable systoles are designed to avoid this issue by replacing the minimal volume by an asymptotic minimal volume.

More precisely, the stable norm of an integral homology class is obtained by representing large multiples of the class, dividing its minimal volume by the multiplicity, and passing to the limit(cf. \cite[section 5]{Federer}). Heuristically, this operation removes some of the nonlinear behavior of individual cycles. The result is a norm, called the stable norm, on the finite dimensional vector space \(H_k(M;\R)\), and the integral homology classes(modulo torsion) form a lattice inside this normed space. The stable systole is just the stable norm of the shortest nonzero lattice vector. The dual norm on \(H^k(M;\R)\) is the comass norm. Thus stable systolic questions become questions about a lattice and its dual. This is the basic reason stable systoles are accessible to cohomological and convex geometry methods, while ordinary systoles often are not. The classical theory begins with this observation. 

The first optimal stable systolic inequality is for complex projective spaces \cite[Section 4.4]{gromov-metric} and relies on cup product and Wirtinger inequality. For rank greater than one, the work \cite{bangert-katz} used cup products together with a successive minima argument in the spirit of Minkowski to prove curvature free inequalities for stable systoles. In the successive minima argument one finds several short independent vectors in the dual cohomology lattice such that they also generate the fundamental cohomology class via cup product. In an elegant recent work \cite{Stryker}, the short independent classes framework of \cite{GHK} is combined with index theory to establish stable \(2\)-systole bounds under positive scalar curvature for a class of manifolds\footnote{The results of \cite{Stryker} apply to all closed spin $2m$-dimensional manifolds whose fundamental cohomology class can be written as a cup product of classes in $H^2(M)$; this includes $S^2 \times T^n, (S^2)^m, \CP^{2n+1}$.}, that inlcudes the important and challenging case of \(S^2\times S^2\). A sharp result in the rank one case of $\CP^n$ is obtained in \cite{CHZ}. The present work grew out of an attempt to understand the results of \cite{Stryker} for \(S^2\times S^2\). 

\subsection{Overview of strategy} The idea is simple and is motivated by a desire to replicate Gromov's proof of the optimal systolic inequality for $\CP^n$ in higher ranks. In that argument the top cup power of the \textit{single} generator gives the fundamental class. To mimic this in higher ranks we introduce a different way of choosing the cohomology class. Instead of first finding a collection of short independent classes\cite{bangert-katz,GHK} with a nontrivial cup product, we prescribe a residue class mod \(n\) that obstructs the vanishing of the cup product or relevant index, and then find a short integral lift of that class. This point of view is particularly effective in rank \(2\), where we can compute the sharp value for the smallest possible length of the integral lift. Although this works in general for all \(2\)-essential or \(2k\)-essential manifolds by choosing an appropriate \(n\), we do not emphasise general cases because we are able to show effective improvements only in the rank two case. 

We now illustrate this through the main example that we are interested in: \(S^2\times S^2\). Let \(x,y\in H^2(S^2\times S^2;\Z)\) be the standard generators, so that \(x^2=y^2=0\) and \(xy\neq0\). If \(c=ax+by\), then \(c^2=2ab\,xy\). Thus \(c^2\neq0\) exactly when both coefficients are nonzero. The idea described above is the trivial observation that one way to enforce $c^2 \neq 0$ is to require \(c\equiv x+y \pmod 2.\) as then \(a\) and \(b\) are both odd, in particular non-zero. This is also the condition required for non-zero index, since \(\Ah(S^2\times S^2)=1\), the twisted Dirac index of a unitary line bundle with Chern class \(c\) is \(\ip{\frac{c^2}{2}}{[S^2\times S^2]} \neq 0 \Leftrightarrow c^2 \neq 0.\) 

Therefore, in this case, both the index argument and the cup product argument only require an integral lift of the mod \(2\) residue class represented by \((1,1)\). We will get a sharp estimate on the length of its integral lifts in terms of a covering constant \(\Theta_{2,2}\) (cf. Definition \ref{def:theta}). We will show in Section \ref{sec:theta-two} that $\Theta_{2,2}=2$.

Informally, \(\Theta_{b,n}\) is the smallest universal constant with the following property: in every rank \(b\) normed lattice, every residue class modulo \(n\) in the dual lattice has a representative whose dual norm is at most \(\Theta_{b,n}\) divided by the first minimum (cf. Definition \ref{def:first-minimum}) of the original lattice. Applied to $H_k(M;\Z)/\tors$ with the stable norm, this means that every prescribed mod \(n\) cohomology class has an integral lift with comass bounded by $\Theta_{b,n}/\stsys_2$ (cf. Proposition \ref{lem:basic-lift}).

Thus estimating $\Theta_{b,n}$ is a finite covering problem. Recall that the ordinary covering radius measures how much one must dilate a convex body so that its lattice translates cover the whole vector space(cf. Definition \ref{def:coverrad}). Here we only ask to hit the finite set of residue classes modulo \(n\)(cf. Definition \ref{def:modn-rad}). If a mod \(n\) class is not represented inside a given dilation of the polar body(cf. Definition \ref{def:polar}), then the corresponding translate of the dilated polar body contains no dual lattice point. Therefore a failure of mod \(n\) covering produces a lattice point free translate(cf. Corollary \ref{cor:free}). Lattice width estimates enter at this point. A basic principle in convex geometry is that a lattice point free convex body must be thin in some integral direction; like a pancake. Upper bounds of this nature are called lattice width estimates. In general proving a lattice width estimate for a general convex body that is lattice point free is a very difficult problem \cite{KL88,Banaszczyk1,Banaszczyk2,aw12}. 

However, for our application here, we will only need to consider a special case where the convex body is centrally symmetric about at a half lattice point. In rank two this will allow us to give an elementary proof using parity arguments without appealing to general results in \cite{aw12} (cf. Lemma \ref{lem:half-lattice-width}). This is the key step in establishing \(\Theta_{2,2}=2\). The constant is sharp, as shown by the \(\ell^1\)-unit ball \(\{(u,v): |u|+|v|\le1\}\) and the \((1,1)\) class modulo \(2\). In higher rank, the \(\ell^1\)-ball gives the lower bound \(\Theta_{b,2}\ge b\). We expect the sharp statement to be \(\Theta_{b,2}=b\) but our current argument doesn't completely generalise to rank $>2$ and the best bound \(\Theta_{b,2}=O(b\log b)\) follows from general results in \cite{Banaszczyk2}.  We leave the sharp width estimate in this special case as a question. 

\begin{question}
    Let \(K\subset\R^b\) be a convex body that is centrally symmetric about a point \(z\in\frac12\Z^b\). Suppose \(\operatorname{int}(K)\cap\Z^b=\varnothing\). Show that the lattice width of \(K\) is at most \(b\). This would imply \(\Theta_{b,2}=b\)
\end{question}
Note that Lemma \ref{lem:half-lattice-width} gives an affirmative answer to this question for $b=2$.

\subsection{Curvature free applications}
Suppose \(M^{2m}\) has degree two classes whose cup product detects the fundamental class. As explained earlier the successive minima method finds \(m\) short independent classes. In contrast, the mod \(n\) method gives one degree two class whose \(m\)-th cup power is already nonzero. For \(\CP^m\), this simply recovers Gromov's sharp stable inequality. For the product \((S^2)^m\) we get the following. 

\begin{theorema} \label{thmA}
For every metric on \((S^2)^m\),
\[
\stsys_2((S^2)^m,g)
\le
\Theta_{m,2}\big[\Vol_g((S^2)^m)\big]^{\frac{1}{m}}.
\]
For \(m=2\), since \(\Theta_{2,2}=2\), this gives
\[
\stsys_2(S^2\times S^2,g)
\le
2\sqrt{\Vol_g(S^2\times S^2)}.
\]
\end{theorema}

The successive minima method gives a constant of the form \((m!)^{\frac{1}{m}}\Gamma_m\) \cite{GHK}. In rank two, this reduces to \(\sqrt2\Gamma_2=3/\sqrt{2}\), since \(\Gamma_2=3/2\)\cite{GHK}. Therefore the constant improves from \(3/\sqrt2\) to \(2\). The expected sharp constant is $1$ for the product of round spheres.

Using Poincare duality like \cite{hebda} we can give show the same result for all closed oriented four manifolds with \(b_2=2\), not only for \(S^2\times S^2\). Choose any nonzero mod \(2\) cohomology class and consider the Poincare dual of its short integral lift. Estimating the mass of the resulting current gives a complementary stable systolic inequality. In dimension four it gives the following.

\begin{theorema} \label{thmB}
Every closed oriented four manifold with \(b_2(M)=b>0\) satisfies
\[
\stsys_2(M,g)
\le
\big(2\Theta_{b,2}\big)^{\frac{1}{2}}\,\sqrt{\Vol_g(M)}.
\]
In particular, for \(b_2(M)=2\) since \(\Theta_{2,2}=2\),
\[
\stsys_2(M,g)
\le
2\sqrt{\Vol_g(M)}.
\]
\end{theorema}

\subsection{Positive scalar curvature applications} In the recent years there have been several works on systolic inequalities in positive scalar curvature \cite{richard,zhu,xu,Sha2026,Orikasa2025}.  The final part of this paper studies stable two systoles under positive scalar curvature lower bounds as in the recent work \cite{Stryker}. The basic idea is as follows:

First, a large stable two systole produces almost flat line bundles. Indeed, a large stable two systole means that the first minimum of the stable norm on \(H_2(M;\Z)/\tors\) is large. By stable norm and comass duality, and by the estimate on mod $n$ covering radius, any residue class mod \(n\) in \(H^2(M;\Z)\) has a representative of small comass. Since integral degree two classes occur as first Chern classes of complex line bundles, Chern--Weil theory realizes such classes as curvature forms of unitary connections. Since, the comass norm is small, the curvature of the corresponding line bundle is also small. 
Second, positive scalar curvature obstructs almost flat index detecting line bundles. Gromov points out in his Four Lectures \cite[Section 3.14.1]{four} that almost flat bundles constrain positive scalar curvature when they are detected by the twisted index. Indeed, if the twisting curvature is sufficiently small compared with the scalar curvature lower bound, then in the Lichnerowicz formula the scalar curvature term dominates the twisting curvature term. The twisted Dirac operator then has trivial kernel and hence zero index, contrary to the index detecting assumption.

Quantitatively, the argument consists of two inequalities. The first is a cowaist scalar curvature inequality, giving an upper bound for line bundle cowaist under a positive scalar curvature lower bound. The second is a cowaist systole inequality, giving a lower bound for line bundle cowaist in terms of the stable two systole. We prove a sharp refinement of cowaist scalar curvature inequality for line bundles (cf. Proposition \ref{prop:refined-cowaist}) and combine it with a cowaist systole inequality (cf. Proposition \ref{prop:theta-cowaist-systole}) obtained involving mod \(n\) covering constant. Therefore one obtains an explicit upperbound on the stable two systole in terms of the covering constant and scalar curvature.

\begin{theorem*}
    Let $(M^{2m},g)$ be a closed spin $2$-essential manifold with $b_2(M)=b>0$ and $\Scal_g \geq \sigma$. Then we have 
\[ \stsys_2(M,g) \leq \frac{8\pi m}{\sigma} \Theta_{b,n}.\]
\end{theorem*}
\begin{remark}
    We note here that this Theorem will be proved in section 6 as Theorem \ref{cor:theta-scalar-general} where instead of $2$-essential, we state it with the condition that there is a residue class mod $n$ that is $\Ah$-detecting. This is not an issue because it is straightforward to show that every $2$-essential spin manifold has an $\Ah$-detecting mod $n$ for some $n$. However, such an $n$ is not universal. The appropriate $n$ will be clear from the manifold. For product of two spheres we will work with $n=2$, i.e., mod $2$ and for complex projective space $\CP^m$ the natural $n$ is $(m+1)$. 
\end{remark}
In particular it gives the following for products of two spheres working mod $2$.
\begin{theorema} \label{thmC}
Let \(g\) be a Riemannian metric on \((S^2)^m\) with \(\Scal_g\ge2m\). Then,
\[
\stsys_2((S^2)^m,g)
\le
4\pi\,\Theta_{m,2} = O(m\log m) .
\]
For $m=2$, \(\Theta_{2,2}=2\). Thus every metric on \(S^2\times S^2\) with \(\Scal_g\ge4\) satisfies
\[
\stsys_2(S^2\times S^2,g)
\le
8\pi.
\]
\end{theorema}

This improves the \(O(m^4 \log m)\) and \(36\pi\) bounds obtained in \cite{Stryker}. In an earlier version of this article \cite{Kumar2026} we had improved them to $O(m^3 \log m)$ and $12\pi$. The current $8\pi$ bound still fails to be sharp by a factor of $2$ as the expected sharp bound is \(4\pi\) for the product of two unit round spheres. 
\begin{remark}
    We note that both the inequalities that we used: the cowaist scalar curvature inequality as well as the cowaist systole inequality are sharp here but we still don't get the sharp $4\pi$ result. We also emphasise that the curvature free systolic constant for $S^2 \times S^2$ also fails to be sharp by a factor of $2$. 
\end{remark}

In the easier rank one case, one gets a sharp two systolic inequality for odd dimensional complex projective spaces. Note that for $\CP^m$ we will have to work mod $(m+1)$; here one trivially has $\Theta_{1,m+1} = \frac{m+1}{2}$ by Proposition \ref{prop:theta-12}) as $m$ is odd.
\begin{theorema}\label{thmD}
Let $m$ be odd. If $g$ is a Riemannian metric on $\CP^m$ satisfying \(\Scal_g\ge4m(m+1)\) then 
\[   \stsys_2(\CP^m,g)\le \frac{2\pi}{m+1} \Theta_{1,m+1} =  \pi. \]
\end{theorema}
\begin{remark}
   A straightforward extension of the $\Ah$-cowaist to determinant line bundles of spin$^c$-manifolds gives the following.
   \begin{enumerate}
       \item  For all closed oriented four manifolds with scalar curvature lower bound $\min \Scal_g \geq \sigma > 0$, \[\stsys_2(M,g) \leq \frac{8\pi}{\sigma} \Theta_{b_2(M),6} <C\frac{b_2(M)(1+\log(b_2(M))}{\sigma}.\]
Here one needs to take covering mod 6 as one needs to ensure that the spin$^c$ index is nonzero; a universal way to do this is by taking a mod $3$ class. However, the constant is not sharp even in the $b_2=1,2$ case. On a case by case basis one can get better constant by optimising over $n$ instead of working with universal $6$ but that is somewhat unsatisfactory. Nonetheless this may be seen as a scalar curvature counterpart of Theorem \ref{thmB}.
\item  This also allows extending Theorem \ref{thmD} to even dimensional complex projective spaces as well with the existing proof \textit{mutatis mutandis} because the reason behind the odd assumption is that even dimensional complex projective spaces are not spin. This is done in the recent work \cite{CHZ}.

   \end{enumerate} 
\end{remark}

\subsection{Organisation} The article has two parts. Part A contains sections 2 and 3, and is devoted to convex geometry. In section 2 we recall basic notions in convex geometry, define the mod $n$ covering constants $\Theta_{b,n}$,  and deduce their basic properties.  In section 3, we have an extended discussion on lattice width estimates for lattice point free convex bodies, and finally prove a sequence of lemmas that culminate in the proof of $\Theta_{2,2}=2$. Part B is devoted to stable systoles and contains section 4,5 and 6. In section 4 we recall the relevant preliminaries and show how stable systoles control lifts of mod $n$ residue classes. In section 5 we prove our our curvature free results, Theorem \ref{thmA} and Theorem \ref{thmB}. In section 6 we prove the sharp cowaist inequality and prove the results under a positive scalar curvature lower bounds, Theorem \ref{thmC} and Theorem \ref{thmD}.  

\subsection{Acknowledgements} The author would like to thank Balarka Sen for several helpful conversations. He would also like to thank Kobe Marshall-Stevens and Simone Cecchini for interest in this work.

\section{Mod $n$ covering constants}
\label{sec:theta}

\subsection{Convex geometry preliminaries} Throughout this section, \(V\) will denote a $b$-dimensional real vector space, \(\Lambda\subset V\) will denote a lattice, and let \(K\subset V\) will denote a centrally symmetric convex body. We will use $\norm{.}_{K}$ to denote the norm on $V$ with $K$ as the unit ball. We first define some basic notions from convex geometry and illustrate them through the simple but important example of $(\R^2,\Z^2)$.

\begin{definition}[Polar Body] \label{def:polar}
    Let $V^*$ be the dual of $V$. Then \(K^\circ=\{\alpha\in V^\ast:|\alpha(v)|\le1 \quad \forall v\in K\}\)  is the polar body of $K$,
\end{definition}
\begin{exam}
Let \((V,\Lambda)=(\R^2,\Z^2)\). Identify \(V^*\) with \(\R^2\) using the standard pairing. Let \(B_1^p\) denote the unit ball of the \(\ell^p\)-norm, and let \(q\) be the conjugate exponent, $1/p +1/q =1$. Then, we claim that \( (B_1^p)^\circ=B_1^q. \) 

Indeed, if \(\alpha\in B_1^q\), then for every \(v\in B_1^p\) by Holder's inequality, \(|\alpha(v)|\le \norm{\alpha}_{\ell^q}\norm{v}_{\ell^p}\le 1.\) Hence \(B_1^q\subset (B_1^p)^\circ\). Conversely, if \(\alpha\notin B_1^q\), then \(\norm{\alpha}_{\ell^q}>1\), and by duality of the \(\ell^p\)- and \(\ell^q\)-norms there exists \(v\in B_1^p\) such that \(|\alpha(v)|>1\). Thus \(\alpha\notin (B_1^p)^\circ\), proving the reverse inclusion. In particular, the polar of the \(\ell^1\)-unit ball is the \(\ell^\infty\)-unit ball, and the polar of the \(\ell^\infty\)-unit ball is the \(\ell^1\)-unit ball.
\end{exam}

\begin{definition}[First minimum]\label{def:first-minimum}
    Given $(V,\Lambda,K, \norm{\cdot}_{K})$ the first minimum is the length of the smallest nonzero element of $\Lambda$, i.e.,  
    \(\lambda_1(K,\Lambda)=\min\{\norm{a}_K:0\neq a\in\Lambda\}\)
\end{definition}

\begin{definition}[Covering radius]\label{def:coverrad}
    The covering radius of \(K\) with respect to \(\Lambda\) is the smallest dilation factor $r$ such that if $rK$ is translated by each vector in $\Lambda$ then it covers all of $V$,
\[
\mu(K,\Lambda)=\inf\{r>0:rK+\Lambda=V\}.
\]
\end{definition}
\begin{exam}
Let \((V,\Lambda)=(\R^2,\Z^2)\), and let \(B_1^p\) denote the unit ball of the \(\ell^p\)-norm. It suffices to look at the fundamental domain \([0,1]^2\). For a point \(x=(x_1,x_2)\in[0,1]^2\), its distance to the nearest lattice point is \(\left(
\min\{x_1,1-x_1\}^p+
\min\{x_2,1-x_2\}^p
\right)^{1/p}.\) 

This quantity is maximized at \(x=(1/2,1/2)\). Therefore \(\mu(B_1^p,\Z^2)
=
\norm{(1/2,1/2)}_{\ell^p}
= 2^{\frac1p-1}.\) 
\end{exam}
\subsection{mod $n$ coverings} We now define a mod $n$ version of covering radius as follows.

\begin{definition}[mod $n$ covering radius]\label{def:modn-rad}
    For \(n\ge2\), define the mod \(n\) covering radius of \(K\) with respect to \(\Lambda\) is the smallest dilation factor \(r\) such that \(rK\) contains at least one representative of every class in \(\Lambda/n\Lambda\). Equivalently,
\[
\mu_n(K,\Lambda)
=
\max_{\rho\in\Lambda/n\Lambda}
\min_{\substack{\lambda\in\Lambda\\ \lambda\equiv\rho\pmod n}}
\norm{\lambda}_K.
\]
\end{definition}
\begin{exam}
    Let $(V,\Lambda)=(\R^2,\Z^2)$ and let $B_1^p$ denote the unit ball under the $\ell^p$-norm. Note that $\Z^2/2\Z^2$ consists of four classes mod $2$: $\{(0,0),(0,1),(1,0),(1,1)\}$ and for every $p$, the first three classes are already in $B_1^p$. Therefore, \( \mu_2(B_1^p,\Z^2) = \norm{(1,1)}_{\ell^p}=2^{\frac{1}{p}}.\) 
\end{exam}

 We can now define the rank $b$ mod $n$ covering constant.
\begin{definition}[Rank \(b\) mod \(n\) covering constant]\label{def:theta}
For \(b>0\) and \(n\ge2\), define
\[
\Theta_{b,n}
=
\sup_{(V,\Lambda,K)}
\lambda_1(K,\Lambda)\,
\mu_n(K^\circ,\Lambda^\ast),
\]
where the supremum is taken over all rank \(b\) normed lattices.
\end{definition}
\begin{remark}\label{rem:normalise}
Observe that if replace \(K\) by \(\lambda_1(K,\Lambda)K\) then it does not change the product \(\lambda_1(K,\Lambda)\mu_2(K^\circ,\Lambda^\ast)\), so estimate $\Theta_{b,n}$ one may assume \(\lambda_1(K,\Lambda)=1\) and try to estimate just \(\mu_n(K^\circ,\Lambda^\ast)\). 
\end{remark}

Unwinding the definition gives the following form that we will frequently use. 
\begin{lemma}\label{lem:theta-unwound}
Consider \((V,\Lambda,K)\). For every class \(\rho\in\Lambda^\ast/n\Lambda^\ast\), there exists an integral lift \(\alpha\in\Lambda^\ast\) with \(\alpha\equiv\rho\pmod n\) and
\[
\norm{\alpha}_{K^\circ}
\le
\frac{\Theta_{b,n}}{\lambda_1(K,\Lambda)}.
\]
\end{lemma}

\begin{proof}
This is immediate from Definition \ref{def:modn-rad} and Definition \ref{def:theta}.
\end{proof}
\subsection{$\Theta_{1,n}=\lfloor n/2\rfloor$}

The following is a direct consequence of the definitions. 
\begin{proposition}\label{prop:theta-12}
\(\Theta_{1,n}=\lfloor n/2\rfloor\).
\end{proposition}
\begin{proof}
First assume that $n=2$, by Remark \ref{rem:normalise} we may assume that $\lambda_1=1$.  Then a generator of \(\Lambda^\ast\) has dual norm \(1\). The two mod \(2\) classes  in \(\Lambda^\ast\) may be represented by \(0\) and \(1\) and both are already contained in the unit ball. More generally, the same argument gives \(\Theta_{1,n}=\lfloor n/2\rfloor\) since the mod $n$ classes are $(0, \pm1, \cdots, \pm \lfloor n/2\rfloor)$. 
\end{proof}
This will be used in the proof of Theorem \ref{cor:CP-scalar} as for $\CP^m$, it will be natural to work mod $(m+1)$ with the residue class represented by $\lfloor (m+1)/2 \rfloor$.  

\section{$\Theta_{2,2}=2$}
\label{sec:theta-two} In this section we will compute the covering constant $\Theta_b,n$ in the first non-trivial case, i.e., $\Theta_{2,2}=2$. 

\subsection{Lattice width estimates}

If a convex body avoids all lattice points, then it cannot be thick in every lattice direction as it would be too large to fit between lattice points. Hence it must ``look flat'' in at least one direction, i.e., there is some nonzero integral linear functional whose values on the body lie in a short interval. To make this precise, the notion of lattice width of a convex body $K$ with respect to a lattice \(\Lambda\) was introduced in \cite{KL88} and is defined as follows.

\begin{definition}[Lattice width]
The lattice width of \(K\) with respect to \(\Lambda\) is the minimum of its width along vectors in the dual lattice \(\Lambda^\ast\),
\[
w(K,\Lambda)
=
\min_{0\neq u\in\Lambda^\ast}
\left(
\max_{x\in K}u(x)-\min_{x\in K}u(x)
\right).
\]
\end{definition}

Thus \(w(K,\Lambda)\) measures how close \(K\) is to being trapped between two nearby parallel lattice hyperplanes.\footnote{Lattice width is distinct from the classical notion of width as only directions coming from the dual lattice are allowed.} The connection with mod \(n\) covering is mediated by two elementary observations. The first relates the first minimum of a normed lattice to the lattice width of the polar body.

\begin{lemma}\label{lem:polar-width}
Let \(V\) be a finite dimensional real vector space, let \(\Lambda\subset V\) be a lattice, and let \(K\subset V\) be an origin symmetric convex body. Then
\[
w(K^\circ,\Lambda^\ast)=2\lambda_1(K,\Lambda).
\]
\end{lemma}

\begin{proof}
We identify \((\Lambda^\ast)^\ast\) with \(\Lambda\). By definition, \(w(K^\circ,\Lambda^\ast)\) is the minimum, over all nonzero \(v\in\Lambda\), of the width of \(K^\circ\) in the direction \(v\). Since \(K^\circ\) is origin symmetric, this width is \(2h_{K^\circ}(v)\), where \(h_{K^\circ}\) is the support function of \(K^\circ\). By polarity, \(h_{K^\circ}(v)=\norm{v}_K\). Hence \(w(K^\circ,\Lambda^\ast)=2\min_{0\neq v\in\Lambda}\norm{v}_K=2\lambda_1(K,\Lambda)\).
\end{proof}

The second observation converts a missing representative of a mod $n$ residue class into a lattice point free translate.

\begin{lemma}\label{lem:residue-translate}
Let \(K\subset V\) be a convex body, let \(\Lambda\subset V\) be a lattice, let \(n\ge2\), and let \(a\in\Lambda\). The mod $n$ class \((a+n\Lambda)\) intersects \( rK\) if and only if \(-a/n+(r/n)K\) contains a point of \(\Lambda\). 
\end{lemma}
\begin{proof}
The class \(a+n\Lambda\) meets \(rK\) if and only if there exists \(\lambda\in\Lambda\) such that \(a+n\lambda\in rK\). This is equivalent to \(\lambda\in -a/n+(r/n)K\). 
\end{proof}
\begin{corollary} \label{cor:free}
\((a+n\Lambda)\cap rK=\varnothing\), then \(-a/n+(r/n)K\) is lattice point free.
\end{corollary}

Together these observations give the mechanism to control mod $n$ radius by the lattice width. By remark \ref{rem:normalise} we may normalise \(\lambda_1(K,\Lambda)=1\). Lemma~\ref{lem:polar-width} gives \(w(K^\circ,\Lambda^\ast)=2\). If a class \(a+n\Lambda^\ast\) fails to meet \(rK^\circ\), then Corollary~\ref{cor:free} gives a lattice point free translate \(-a/n+(r/n)K^\circ\). This translate has width \(2r/n\). Thus any lattice width estimate for lattice point free translates gives an upper bound on \(r\), i.e., gives an upper bound on the mod $n$ covering radius $\mu_n(K^\circ,V^*)$. Since, we normalised $\lambda_1(K,\Lambda)=1$, this is equivalent to an upper bound on $\Theta_{b,n}$.

In general, proving sharp lattice width estimate for arbitrary lattice point free convex bodies is a difficult question \cite{KL88}. The current best estimate is $O(b\log b)$ by the work \cite{Banaszczyk2}. In rank two, a sharp estimate is proved in \cite{aw12} fixing a gap in the earlier work \cite{KL88}. However, this rank two result is much more general than what is needed here. The proof is quite long and uses several deep results along with unavoidable delicate case by case arguments. 

Fortunately, in our required rank two mod \(2\) application, the situation is greatly simplified as the lattice point free translate produced by Lemma~\ref{lem:residue-translate} is centered at a half lattice point. This allows us to use parity arguments to give an elementary proof of the required lattice width estimate. Using that, we will show that \(\Theta_{2,2}=2\).

\subsection{Lattice width in rank two}

\begin{lemma}\label{lem:half-lattice-width}
Let \(K\subset\R^2\) be a convex body centrally symmetric about a point \(z\in\frac12\Z^2\). Suppose \(\operatorname{int}(K)\cap\Z^2=\varnothing\). Then the lattice width of \(K\) is at most \(2\).
\end{lemma}

\begin{proof}
It is convenient to split the argument in five steps. 

\emph{Step 1: K is a bounded maximal lattice point free set.}
Since \(K\) is lattice point free, so \(z\notin\Z^2\). Thus \(z\) represents a nonzero element of \(\frac12\Z^2/\Z^2\). After a unimodular affine map\footnote{This means that the map preserves the lattice $\Z^2$. Concretely consider $x \to Ax + b$, for $A \in GL(2;Z)$ with $\det(A)= \pm 1$ and $b \in \Z^2$}, we may assume \(z=(1/2,1/2)\). This does not change lattice width.

Enlarge \(K\), keeping the same center \(z\), to a maximal centrally symmetric, interior lattice point free convex set that that we still call $K$. Since enlarging cannot decrease lattice width, it is enough to prove the estimate for such a maximal enlargement. If this maximal set is unbounded, then by central symmetry it is a strip. A maximal lattice point free strip lies between two adjacent parallel lattice lines, and therefore has lattice width \(1\). Hence we may assume from now on that the \(K\) is bounded.

\emph{Step 2: $K$ is a centrally symmetric parallelogram.}

We first note that all boundary segments of $K$ are straight line segments because any curved boundary portions can be pushed slightly outwards without creating new lattice points, contradicting maximality. Next observe that every side of $\partial K$ must contain a lattice point in its relative interior; otherwise the side, and by symmetry the opposite side, could be moved slightly outward without creating an interior lattice point, contradicting maximality. 

Hence, $K$ is a polygon and every side contains a lattice point in its relative interior. Since $K$ is bounded, $\partial K$ contains only finitely many lattice points, and hence it has only finitely many sides. Note that by central symmetry there must be an even number of sides. Choose one lattice point in the relative interior of each side. If the polygon had at least six sides, then among these chosen lattice points two would have the same parity mod \(2\). But these are non-opposite sides as those have different parity mod $2$ since $2z=(1,1)$. But for non-opposite sides with same parity, the midpoint of line joining them would be an interior lattice point for $K$, contradiction. Therefore \(K\) has at most four sides. Since \(K\) is bounded and centrally symmetric, it must have exactly four sides. Thus \(K\) is a parallelogram.

\emph{Step 3: An inscribed fundamental parallelogram.}
Let \(p_1\) and \(p_2\) be lattice points in the relative interiors of two adjacent sides of \(K\). Since \(K\) is centrally symmetric about \(z=(1/2,1/2)\), the points \(p_3=2z-p_1,\quad p_4=2z-p_2\) are lattice points in the relative interiors of the opposite sides. Let $P$ be the lattice parallelogram contained in \(K\) with lattice point vertices $p_i$. The interior of the edges of \(P\) contain no lattice points as they lie in interior of $K$. Similarly $P$ also has no interior lattice point. Therefore by Pick's formula, the area of $P$ is
\[
\operatorname{area}(P)=0+\frac{4}{2}-1=1.
\]
Hence \(P\) is a fundamental parallelogram for the lattice. After a unimodular change of coordinates, we may assume that \(P=[0,1]^2,\) centered at \((1/2,1/2)\).

\emph{Step 4: Describing $K$ as a set.}
We now just need to prove the width estimate for a parallelogram \(K\) centered at \((1/2,1/2)\), circumscribed around the unit square, with one vertex of the square on each side of \(K\). Shift the center to the origin. The square $P$ now has vertices $-v,v,w,-w$ where $v=(1/2,1/2)$ and $w=(1/2,-1/2)$. 

As $K$ circumscribes $P$, the two pairs of opposite sides of \(K\) pass through \(\pm v\) and \(\pm w\). Use $(v,w)$ coordinates to write points as \(x=sv+tw\). Then the vertices of the square are \((\pm1,0)\) and \((0,\pm1)\). Since one pair of sides of \(K\) passes through \((1,0)\) and \((-1,0)\), that pair is given, after normalization, by \(|s+at|\le1\) for some real number \(a\). Since the other pair of sides passes through \((0,1)\) and \((0,-1)\), it is given by \(|bs+t|\le1\) for some real number \(b\). The fact that \(K\) contains all four vertices of the square implies \(|a|\le1\) and \(|b|\le1.\) Thus we can write $K$ as the following set \[K = \{(s,t): |s+at| \leq 1, |bs+t| \leq 1; \quad a,b \in [-1,1].\]

\emph{Step 5: Width calculation.}
Set \(D=1-ab\). Since \(K\) is bounded and \(|a|,|b|\le1\), we have \(D>0\). Let \(p=s+at\) and \(q=bs+t\). Thus \(K\) is given by \(|p|\le1\) and \(|q|\le1\).  Solving for \(s,t\) gives \(s=(p-aq)/D\) and \(t=(-bp+q)/D\).

Now return to the original lattice coordinates. Since \(v=(1/2,1/2)\) and \(w=(1/2,-1/2)\), a point \(x=sv+tw\) has coordinates \(X=(s+t)/2\) and \(Y=(s-t)/2\). Substituting the formulas for \(s\) and \(t\), we get
\[
X=\frac{(1-b)p+(1-a)q}{2D},
\qquad
Y=\frac{(1+b)p-(1+a)q}{2D}.
\]
As \(p\) and \(q\) vary independently in \([-1,1]\), and since \(|a|,|b|\le1\), the width of \(K\) in the \(X\)-direction is \(((1-b)+(1-a))/D=(2-a-b)/D\). Similarly, the width of \(K\) in the \(Y\)-direction is \(((1+b)+(1+a))/D=(2+a+b)/D\).

Let \(S=a+b\). The smaller of these two widths is \((2-|S|)/D\). We claim this is at most \(2\). Since \(D=1-ab\), this is equivalent to \(2-|a+b|\le2(1-ab)\), or equivalently \(|a+b|\ge2ab\).

This is trivially true if \(ab\le0\). If \(ab>0\), then \(a\) and \(b\) have the same sign, so \(|a+b|=|a|+|b|\). Since \(|a|,|b|\le1\), we have
\[
|a|+|b|-2|a||b|
=
|a|(1-|b|)+|b|(1-|a|)
\ge0.
\]
Thus \(|a+b|\ge2|a||b|=2ab\).

Therefore either the \(X\)-width or the \(Y\)-width of \(K\) is at most \(2\). Since \(X\) and \(Y\) are the two lattice coordinates, the lattice width of \(K\) is at most \(2\).
\end{proof}

This gives us the following useful corollary.

\begin{corollary}\label{cor:width-two-mod-two-covering}
Let \(\Lambda\) be a rank two lattice and let \(K\) be a centrally symmetric convex body centered at the origin with \(w(K,\Lambda)\ge2\), then every class in \(\Lambda/2\Lambda\) has a representative in \(2K\). Equivalently, \(\mu_2(K,\Lambda)\le2\).
\end{corollary}

\begin{proof}
Fix a class \(a+2\Lambda\). Suppose, for contradiction, that \((a+2\Lambda)\cap2K=\varnothing\). Since \(2K\) is compact and \(a+2\Lambda\) is closed and discrete, there exists \(\varepsilon>0\) such that \((a+2\Lambda)\cap2(1+\varepsilon)K=\varnothing\). By Corollary~\ref{cor:free}, the translate \(-a/2+(1+\varepsilon)K\) contains no point of \(\Lambda\). This translate is centrally symmetric about the half lattice point \(-a/2\), and its lattice width is \((1+\varepsilon)w(K,\Lambda)\ge2(1+\varepsilon)>2\). This contradicts Lemma~\ref{lem:half-lattice-width}. Hence every class in \(\Lambda/2\Lambda\) has a representative in \(2K\).
\end{proof}

We can now compute the rank two mod \(2\) covering constant.

\begin{proposition}\label{prop:theta-22}
\(\Theta_{2,2}=2\).
\end{proposition}

\begin{proof}
We first prove the upper bound. Let \((V,\Lambda,K)\) be a rank two space. By  remark \ref{rem:normalise} we may assume \(\lambda_1(K,\Lambda)=1\). By Lemma~\ref{lem:polar-width}, \(w(K^\circ,\Lambda^\ast)=2\). Applying Corollary~\ref{cor:width-two-mod-two-covering} to the rank two lattice \(\Lambda^\ast\) and the origin symmetric convex body \(K^\circ\), every class in \(\Lambda^\ast/2\Lambda^\ast\) has a representative in \(2K^\circ\). Therefore \(\mu_2(K^\circ,\Lambda^\ast)\le2\). Taking the supremum over all rank two normed lattices gives \(\Theta_{2,2}\le2\), since we had normalized \(\lambda_1(K,\Lambda)=1\).

For the lower bound, take \(V=\R^2\), \(\Lambda=\Z^2\), and \(K=[-1,1]^2\). Then \(\lambda_1(K,\Lambda)=1\), and the polar body is \(K^\circ=\{(u,v): |u|+|v|\le1\}\). Any representative of the class \((1,1)+2\Z^2\) has \(K^\circ\)-norm at least \(2\). Thus \(\mu_2(K^\circ,\Lambda^\ast)\ge2\), so \(\Theta_{2,2}\ge2\). Combining the two inequalities gives \(\Theta_{2,2}=2\).
\end{proof}

\section{Stable systole and short lifts}
In this section we will recall some basic notions and definitions from \cite{Federer}. Throughout, \((M,g)\) will be a closed Riemannian manifold. If \(T\) is a $k$-cycle, we will write \(\mathbf M(T)\) for its mass with respect to \(g\). 

\subsection{Stable norms and duality}
Let \(a\in H_k(M;\Z)/\mathrm{tors}\). For each positive integer \(r\), set
\[
m_r(a)=\inf\{\mathbf M(T): T \text{ is an integral } k\text{-cycle representing }ra\}.
\]
The sequence \(m_r(a)\) is subadditive, so the limit under stabilisation exists and gives the stable norm on integral classes, i.e.,
\[
\norm{a}_{\mathrm{st}}=\lim_{r\to\infty}\frac{m_r(a)}{r}.
\]
The stable norm extends by homogeneity to rational classes and then by continuity to a norm on \(H_k(M;\R)\). Hence, $(H_k(M;\R), \norm{\cdot}_{\mathrm{st}})$ is a finite dimensional normed vector space, in particular it is a Banach space.
Now, let \(\omega\) be a smooth \(k\)-form on \(M\). Its comass is
\[
\norm{\omega}_{\infty}
=
\sup_{x\in M} \sup_{\xi}
 \{|\omega_x(\xi)|:\xi\ \text{ is a unit simple $k$-vector in } \wedge^kT_xM \}.
\]
For a cohomology class \(u\in H^k(M;\R)\), its comass norm is
\[
\norm{u}^{\ast}
=
\inf\{\norm{\omega}_{\infty}: \omega \text{ is a closed smooth } k\text{-form representing }u\}.
\]
By Federer duality \cite{Federer}, the Banach spaces $(H_k(M;\R), \norm{\cdot}_{\mathrm{st}})$ and $(H^k(M;\R), \norm{\cdot}^{*})$ are dual to each other. Further, modulo torsion, the integral homology and cohomology are lattices inside $(H_k(M;\R), \norm{\cdot}_{\mathrm{st}})$ and $(H^k(M;\R), \norm{\cdot}^{*})$ respectively and are also dual to each other.  We shall denote them by $\Lambda_k(M)$ and $\Lambda^k(M)$, respectively, i.e., we have
\[
\Lambda_k(M)=H_k(M;\Z)/\tors \quad\text{dual to} \quad \Lambda^k(M)=  H^k(M;\Z)/\tors.
\]

\subsection{Short lifts from stable systoles}

\begin{definition}[Stable \(k\)-systole]
The stable \(k\)-systole of \((M,g)\) is the first minimum(Definition \ref{def:first-minimum}) of the lattice $\Lambda_k(M)$ with respect to the stable norm, i.e.,
\[
\stsys_k(M,g)
= \lambda_1(\Lambda_k(M))=
\min\{\norm{a}_{\mathrm{st}}:0\neq a\in\Lambda_k(M)\}.
\]
\end{definition}

The following proposition will be used several times. It is the mechanism that gives controls the length of integral lifts of a mod $n$ class in terms of the covering constant and the stable systole. It is a direct consequence of applying Lemma~\ref{lem:theta-unwound} to residue classes in \(\Lambda^k(M)/n\Lambda^k(M)\).

\begin{proposition}[Short lifts of mod \(n\) classes]\label{lem:basic-lift}
Let \((M,g)\) be a closed Riemannian manifold with \(b=\rank H^k(M;\Z)>0\). For every class \(\rho\in\Lambda^k(M)/n\Lambda^k(M)\), there exists an integral lift \(c\in\Lambda^k(M)\) with \(c\equiv\rho\pmod n\) and
\[
\norm{c}^{\ast}
\le
\frac{\Theta_{b,n}}{\stsys_k(M,g)}.
\]
\end{proposition}

\begin{proof}
Consider $(H_k(M;\R), \Lambda_k(M), B_1^{\mathrm st})$; here $B_1^{\mathrm{st}}$ is the unit ball in the stable norm. By Federer duality, the unit comass ball is the polar body of the unit stable norm ball. The result now follows from Lemma~\ref{lem:theta-unwound}, as $\lambda_1(B_1^{\mathrm{st}}, \Lambda_k(M)) = \stsys_k(M,g)$.
\end{proof}

\section{Curvature free applications}
\label{sec:curvature-free}

In this section we prove the curvature free applications stated in the introduction,i.e., we prove Theorem \ref{thmA} and Theorem \ref{thmB}.

\subsection{Products of two spheres}

We will use the following standard Wirtinger inequality for powers of two-forms. Here, for a form on a Euclidean vector space, \(\norm{\cdot}_\infty\) denotes the comass norm.

\begin{lemma}\label{lem:wirt-two}
On a Euclidean vector space of dimension \(2m\), every \(2\)-form \(\alpha\) satisfies the sharp inequality
\[
|\alpha^m|
\le
m!\norm{\alpha}_{\infty}^m.
\]
\end{lemma}

\begin{proof}
Write \(\alpha\) in normal form \(\alpha
=
\lambda_1 e^1\wedge e^2
+
\lambda_2 e^3\wedge e^4
+
\cdots+
\lambda_m e^{2m-1}\wedge e^{2m}.\)
Then \(\norm{\alpha}_{\infty}=\max_i|\lambda_i|\), and
\[
\alpha^m=m!\lambda_1\cdots\lambda_m\,dV.
\]
The inequality follows. Equality is attained when all \(|\lambda_i|\) are equal.
\end{proof}
We will also need the following simple observation. This will also be used in section 6. 

\begin{lemma}\label{lem:all-odd-class}
Let \(M=(S^2)^m\), and let \(x_1,\ldots,x_m\in H^2(M;\Z)\) be the standard generators. Then for the all odd mod $2$ residue class $\rho$, every integral lift $c$ satisfies \(c^m \neq 0 \) and \( \left|\langle c^m,[M]\rangle\right|\ge m!.\)
\end{lemma}
\begin{proof}
Take $\rho = x_1 + \cdots + x_m\pmod2$. Then any lift $c$ is given by $ a_1x_1 + \cdots+a_mx_m$ where all $a_i$ are odd, in particular $|a_1 \cdots a_m| \geq 1$. Further, since \(x_i^2=0\) and \(x_1\cdots x_m\) generates \(H^{2m}(M;\Z)\), the conclusion follows as $c^m = m! a_1\cdots a_mx_1\cdots x_m$.
\end{proof}

We now prove Theorem \ref{thmA}.

\begin{theorem}\label{prop:product-spheres-no-curv}
For every metric on \((S^2)^m\),
\[
\stsys_2((S^2)^m,g)
\le
\Theta_{m,2}\big[\Vol_g((S^2)^m)\big]^{\frac{1}{m}}.
\]
In particular,
\[
\stsys_2(S^2\times S^2,g)
\le
2\sqrt{\Vol_g(S^2\times S^2)}.
\]
\end{theorem}

\begin{proof}
Take $\rho$ from Lemma \ref{lem:all-odd-class}. By Proposition~\ref{lem:basic-lift}, it has an integral lift \(c\in H^2((S^2)^m;\Z)\) such that
\[
\norm{c}^{\ast}
\le
\frac{\Theta_{m,2}}{\stsys_2((S^2)^m,g)}.
\]
But by Lemma \ref{lem:all-odd-class}, for we also have  \(c^m \neq 0 \) and \( \left|\langle c^m,[M]\rangle\right|\ge m!.\) Therefore choose a closed \(2\)-form \(\alpha\) representing \(c\) with
\[
\norm{\alpha}_\infty
\le
\frac{\Theta_{m,2}}{\stsys_2((S^2)^m,g)}+\varepsilon.
\]
Combining the above we get,
\[
m!
\le
\left|\int_{(S^2)^m}\alpha^m\right|
\le
m!\left(
\frac{\Theta_{m,2}}{\stsys_2((S^2)^m,g)}+\varepsilon
\right)^m
\Vol((S^2)^m,g).
\]
 Letting \(\varepsilon\to0\) and rearranging we get the first conclusion. For $m=2$, \(\Theta_{2,2}=2\) and therefore the second conclusion also follows. 
\end{proof}

\begin{remark}
The same proof with the hyperplane class gives Gromov's stable inequality for \(\CP^m\),
\[
\stsys_2(\CP^m,g)^m\le m!\Vol(\CP^m,g).
\]
We point this out only to emphasize that the mod \(n\) lifting viewpoint reduces to the classical rank one argument when \(b_2=1\). A new feature for \((S^2)^m\) is that the all odd mod \(2\) class has a top power with an algebraic factor \(m!\), which cancels the \(m!\) in the Wirtinger estimate.
\end{remark}

\subsection{Poincare duality estimate}

Let \(A_{k,d-k}\) be the smallest constant such that, on every Euclidean vector space of dimension \(d\), every \(k\)-form \(\alpha\) and every \((d-k)\)-form \(\beta\) satisfy
\[
|\alpha\wedge\beta|
\le
A_{k,d-k}
\norm{\alpha}_{\infty}
\norm{\beta}_{\infty}
\,dV.
\]

We will use the following well known observation. 
\begin{lemma}\label{lem:A22}
\(A_{2,2}=2\).
\end{lemma}

\begin{proof}
Let \(\alpha\) be a \(2\)-form on a Euclidean \(4\)-space. Write \(\alpha=\lambda_1 e^1\wedge e^2 +\lambda_2 e^3\wedge e^4\) in normal form. If \(\beta\) has comass at most \(1\), then the coefficient of \(\alpha\wedge\beta\) is bounded in absolute value by \(|\lambda_1|+|\lambda_2|\). Since \(\norm{\alpha}_\infty=\max(|\lambda_1|,|\lambda_2|)\), we get
\[
|\alpha\wedge\beta|
\le
2 \norm{\alpha}_\infty \norm{\beta}_\infty \,dV.
\]
Equality is attained by \(\alpha=\beta\).
\end{proof}

The following lemma is the main step in the proof of Theorem \ref{thmB}. 

\begin{lemma}\label{lem:pd-mass-estimate}
Let \(M^d\) be closed and oriented, and let \(c\in \Lambda^k(M)\). Then
\[
\|\operatorname{PD}(c)\|_{\mathrm{st}}
\le
A_{k,d-k}\|c\|^*\operatorname{Vol}(M,g).
\]
\end{lemma}

\begin{proof}
Choose a closed \(k\)-form \(\alpha\) representing \(c\) with \(\|\alpha\|_\infty\le \|c\|^*+\varepsilon.\) For smooth \((d-k)\)-forms \(\beta\) define a \((d-k)\)-current \(T_\alpha\) by
\[
T_\alpha(\beta)=\int_M \alpha\wedge\beta.
\]
 Then $T_{\alpha}$ represents \(\operatorname{PD}(c)\). By the definition of \(A_{k,d-k}\),
\[
M(T_\alpha)
\le
A_{k,d-k}(\|c\|^*+\varepsilon)\operatorname{Vol}(M,g).
\]
Since the stable norm is the mass norm on real homology we have \(\|\operatorname{PD}(c)\|_{\mathrm{st}}\le M(T_\alpha).\) Taking \(\varepsilon\to0\) gives the claim.
\end{proof}

The following complementary stable systolic inequality in particular also provides another proof of the $S^2 \times S^2$ inequality. In particular we get Theorem \ref{thmB}. 

\begin{theorem}\label{prop:PD}
Let \(M^d\) be closed and oriented. If \(b_k(M)>0\), then
\[
\stsys_k(M,g)\stsys_{d-k}(M,g)
\le
A_{k,d-k}\Theta_{b_k,2}\Vol(M,g).
\]
For $d=4$,
\[
\stsys_2(M,g)
\le
(2 \Theta_{b_2,2})^{\frac{1}{2}}\sqrt{\Vol(M,g)}.
\]
For $d=4$ and $b_2=2$,
\[\stsys_2(M,g)
\le
2\sqrt{\Vol(M,g)}.\]
\end{theorem}

\begin{proof}
Choose a nonzero class \(\rho\in \Lambda^k(M)/2\Lambda^k(M)\). By Proposition \ref{lem:basic-lift}, there is an integral lift \(c\in\Lambda^k(M)\) with \(c\equiv\rho\pmod2\) and
\[
\|c\|^*
\le
\frac{\Theta_{b_k(M),2}}{\operatorname{stsys}_k(M,g)}.
\]
Since \(\rho\neq0\), the class \(c\) is nonzero in the free part of cohomology, so \(\operatorname{PD}(c)\) is nonzero in \(H_{d-k}(M;\Z)/\mathrm{tors}\). Hence
\[
\operatorname{stsys}_{d-k}(M,g)
\le
\|\operatorname{PD}(c)\|_{\mathrm{st}}.
\]
Applying Lemma~\ref{lem:pd-mass-estimate} gives
\[
\operatorname{stsys}_{d-k}(M,g)
\le
A_{k,d-k}
\frac{\Theta_{b_k(M),2}}{\operatorname{stsys}_k(M,g)}
\operatorname{Vol}(M,g).
\]
Rearranging gives the desired conclusion. The other conclusions follow from Lemma \ref{lem:A22} and that $\Theta_{2,2}=2$.
\end{proof}


\section{Positive scalar curvature}
\label{sec:positive-scalar-curvature}

In this section we first prove the sharp line bundle cowaist inequality. On combining it with the mod \(n\) lifting estimate we get a stable two systole estimate under positive scalar curvature.

\subsection{Line bundle cowaist}

We use the following convention. Throughout this section, \(E\to M\) will denote Hermitian vector bundles with metric connection. \(\norm{R^E}_\infty\) will denote the supremum of the comass norm in the two form variables together with the operator norm in the bundle direction.

\begin{definition}\label{def:Ah-cowaist}
Let \((M,g)\) be a closed spin Riemannian manifold. A Hermitian vector bundle \(E\to M\) is called \(\Ah\)-admissible if
\[
\ip{\Ah(TM)\ch(E)}{[M]}\neq0.
\]
The \(\Ah\)-cowaist of \(M\) is
\[
\Ahcw(M,g)
=
\sup \left\{
\norm{R^E}_\infty^{-1} :
E \text{ is } \Ah\text{ admissible} 
\right\}.
\]
Restricting to line bundles gives
\[
\Ahcwline(M,g)
=
\sup \left\{
\norm{R^L}_\infty^{-1} :
L \text{ is } \Ah\text{ admissible} 
\right\}.
\]
\end{definition}

For a closed spin manifold of real dimension \(2m\), Gromov's cowaist scalar curvature inequality gives an upper bound for cowaist under a positive scalar curvature lower bound; see \cite[Theorem 5 1/4]{Gromov} and also \cite[Theorem 3.4]{Stryker}. For line bundles this gives a bound of the form
\[
\Ahcwline(M^{2m},g)
\le
\frac{4m(2m-1)}{\min_M\Scal_g}.
\]
We prove the following sharp refinement.
\begin{proposition}\label{prop:refined-cowaist}
Let \((M^{2m},g)\) be a closed spin Riemannian manifold, and let \(\sigma=\min_M\Scal_g>0\). Then
\[
\Ahcwline(M,g)
\le
\frac{4m}{\sigma}.
\]
Equivalently, every \(\Ah\)-admissible Hermitian line bundle \(L\to M\) with unitary connection satisfies
\begin{equation}\label{eq:waistequiv}
\norm{R^L}_\infty \ge \frac{\sigma}{4m}.
\end{equation}
\end{proposition}

\begin{proof}
Let \(L\to M\) be an \(\Ah\)-admissible Hermitian line bundle with a unitary connection. Since the twisted Dirac index is nonzero, the corresponding twisted Dirac operator has a nonzero harmonic spinor \(\psi\).

The Lichnerowicz formula gives
\[
0
=
\int_M
\left(
|\nabla\psi|^2
+
\frac{\Scal_g}{4}|\psi|^2
+
\ip{\mathcal R^L\psi}{\psi}
\right)dV,
\]
where
\[
\mathcal R^L
=
\frac12
\sum_{i,j}c(e_i)c(e_j)R^L(e_i,e_j).
\]
We claim that \(\norm{\mathcal R^L}_{\operatorname{op}}\le m\norm{R^L}_\infty\). Fix a point $x\in M$. Since $L_x$ is one-dimensional and the connection is unitary, write $R^L=i\alpha$ with $\alpha$ a real two-form at $x$. Choose an orthonormal basis putting $\alpha$ in skew-normal form:
\begin{equation}\label{eq:normal-form}
\alpha
=
\lambda_1 e^1\wedge e^2
+
\lambda_2 e^3\wedge e^4
+
\cdots+
\lambda_n e^{2m-1}\wedge e^{2m}.
\end{equation}
For a two-form in this normal form, its comass is $\max_j |\lambda_j|$. With our convention for the line bundle curvature norm, this means\( \norm{R^L}_x=\norm{\alpha}_{\comass}=\max_j|\lambda_j|.\)
In the basis \eqref{eq:normal-form}, the twisting term becomes
\begin{equation*}
\mathcal R^L
=
\sum_{j=1}^m i\lambda_j c(e_{2j-1})c(e_{2j}).
\end{equation*}
Set $B_j=i c(e_{2j-1})c(e_{2j})$. The operators $B_j$ are self-adjoint involutions, and they commute because they involve disjoint Clifford generators. Hence they can be diagonalized simultaneously, with eigenvalues $\pm1$. Therefore the eigenvalues of $\mathcal R^L$ are of the form $\sum_j \pm \lambda_j$. Therefore we have as claimed,
\begin{equation}\label{eq:op-norm-bound}
\norm{\mathcal R^L}_{\operatorname{op}}
\le
\sum_{j=1}^m |\lambda_j|
\le
m\max_j|\lambda_j|
=
m\norm{R^L}_x.
\end{equation}
It follows that
\[
\ip{\mathcal R^L\psi}{\psi}
\ge
-m\norm{R^L}_\infty|\psi|^2.
\]
Using \(\Scal_g\ge\sigma\), we obtain
\[
0
\ge
\int_M
\left(
|\nabla \psi|^2
+
\left(\frac{\sigma}{4}-m\norm{R^L}_\infty\right)
|\psi|^2
\right)dV.
\]
If \(\norm{R^L}_\infty<\sigma/(4m)\), then the integrand is nonnegative and is positive wherever \(\psi\neq0\), a contradiction. Hence \(\norm{R^L}_\infty\ge\sigma/(4m)\). Taking reciprocals and then taking the supremum over \(\Ah\)-admissible line bundles gives the cowaist inequality.
\end{proof}
\begin{remark}
    We note here that the Clifford algebra estimate above for $\norm{\mathcal{R}^L}_{op}$ is also deduced in the proof of Lemma 3.1 in \cite{CHZ}.  
\end{remark}

\subsection{Cowaist systole estimate} Let \(M^{2m}\) be closed and spin. We say that a class \(\rho\in \Lambda^2(M)/n\Lambda^2(M)\) is \(\Ah\)-detecting if every lift \(c\in \Lambda^2(M)\) with \(c\equiv\rho\pmod n\) satisfies
\[
\ip{\Ah(TM)e^c}{[M]}\neq0.
\]
Equivalently, every line bundle \(L\to M\) with \(c_1(L)\equiv\rho\pmod n\) is \(\Ah\)-admissible.

\begin{proposition}\label{prop:theta-cowaist-systole}
Let \((M^{2m},g)\) be a closed spin Riemannian manifold. Suppose that \(\rho\in \Lambda^2(M)/n\Lambda^2(M)\) is \(\Ah\)-detecting. Then
\[
\Ahcwline(M,g)
\ge
\frac{\stsys_2(M,g)}{2\pi\,\Theta_{b_2(M),n}}.
\]
\end{proposition}

\begin{proof}
By Proposition~\ref{lem:basic-lift}, the class \(\rho\) has a lift \(c\in \Lambda^2(M)\) such that
\[
\norm{c}^{\ast}
\le
\frac{\Theta_{b_2(M),n}}{\stsys_2(M,g)}.
\]
Choose a closed two form \(\omega\) representing \(c\) with
\[
\norm{\omega}_\infty
\le
\frac{\Theta_{b_2(M),n}}{\stsys_2(M,g)}+\varepsilon.
\]
Let \(L\to M\) be a Hermitian line bundle with \(c_1(L)=c\). Since \(\rho\) is \(\Ah\) detecting, \(L\) is \(\Ah\)-admissible. Since, $c_1(L)=c$, and $\omega$ is a representative of $c$, \(L\) admits a unitary connection with curvature \(R^L=2\pi i\,\omega\). Hence
\[
\norm{R^L}_\infty
\le
2\pi\left(
\frac{\Theta_{b_2(M),n}}{\stsys_2(M,g)}+\varepsilon
\right).
\]
Letting \(\varepsilon\to0\), and taking reciprocals, gives the claim.
\end{proof}

Combining Proposition \ref{prop:refined-cowaist} and Proposition \ref{prop:theta-cowaist-systole} gives a general two systole upper bound via scalar curvature. 
 
\begin{theorem}\label{cor:theta-scalar-general}
Let \((M^{2m},g)\) be a closed spin Riemannian manifold with \(\Scal_g \geq \sigma >0\) and $b_2(M)=b>0$. Suppose that \(\rho\in \Lambda^2(M)/n\Lambda^2(M)\) is \(\Ah\)-detecting. Then
\[
\stsys_2(M,g)
\le
\frac{8\pi m}{\sigma}\Theta_{b,n}.
\]
\end{theorem}

\begin{proof}
Proposition~\ref{prop:theta-cowaist-systole} gives
\[
\frac{\stsys_2(M,g)}{2\pi\,\Theta_{b,n}}
\le
\Ahcwline(M,g).
\]
By Proposition~\ref{prop:refined-cowaist}, we also have \(\Ahcwline(M,g)\le4m/\sigma\). Combining these two inequalities gives the result.
\end{proof}

We will now deduce Theorem \ref{thmC} and Theorem \ref{thmD} as corollaries of Theorem \ref{cor:theta-scalar-general}.

 \subsection{Products of two spheres}

Let \(M=(S^2)^m\). By Lemma \ref{lem:all-odd-class} every integral lift $c$ of all odd mod $2$ class $\rho$ satisfies $c^m \neq0$. As a corollary we obtain. 

\begin{corollary}
$\rho$ is $\Ah$-detecting. 
\end{corollary}
\begin{proof}
Indeed let $L$ be the line bundle with $c_1(L)=c$ for any lift $c$ of $\rho$.  Since \(\Ah(T(S^2)^m)=1\), by the index pairing defining admissibility we have,
\[
\ip{\Ah(TM)\mathrm{ch}(L)}{[M]}
= \ip{e^c}{[M]} = \frac{1}{m!} \ip{c^m}{[M]}
\neq0.
\]
\end{proof}

Therefore we get Theorem \ref{thmC}.

\begin{theorem}\label{cor:product-spheres-scalar}
Let \(g\) be a Riemannian metric on \((S^2)^m\) with \(\Scal_g\ge2m\). Then
\[
\stsys_2((S^2)^m,g)
\le
4\pi\,\Theta_{m,2}.
\]
For $S^2 \times S^2$ with $\Scal_g \geq 4$ \[
\stsys_2(S^2 \times S^2,g)
\le
8\pi.
\]
\end{theorem}

\begin{proof}
Here \(\dim (S^2)^m=2m\), \(b_2((S^2)^m)=m\). The class $c$ that is a lift of the all odd class $\rho$ mod \(2\) is \(\Ah\)-detecting. Applying Theorem~\ref{cor:theta-scalar-general} with \(\sigma=2m\) and with mod \(2\) gives the desired conclusion. For $m=2$, using $\Theta_{2,2}=2$ gives the second inequality. 
\end{proof}

\subsection{Odd complex projective spaces}

Let \(m\) be odd, so that \(\CP^m\) is spin. Let \(h\in H^2(\CP^m;\Z)\) be the positive generator, normalized by \(\ip{h^m}{[\CP^m]}=1\). We write \(N=m+1\) and \(a=N/2\). We use the class \(ah\pmod N\). 

This class is \(\Ah\)-detecting. Indeed, if \(c\equiv ah\pmod N\), then \(c=(a+N\ell)h\) for some \(\ell\in\Z\). Since \(c_1(T\CP^m)=Nh\), we have \(\Ah(T\CP^m)e^{Nh/2}=\Td(T^{1,0}\CP^m).\) Therefore by the index pairing defining $\Ah$-admissibility,
\[
\ip{\Ah(T\CP^m)e^c}{[\CP^m]}
=
\ip{\Td(T^{1,0}\CP^m)e^{N\ell h}}{[\CP^m]}
=
\chi(\mathcal O_{\CP^m}(N\ell)).
\]
Where the last equality is by Hirzebruch-Riemann-Roch. Since \(\chi(\mathcal O_{\CP^m}(q))=\binom{q+m}{m}\), this vanishes only for \(q=-1,\ldots,-m\). In particular it never vanishes for \(q=N\ell=(m+1)\ell\). Hence every lift of \(ah\pmod N\) is \(\Ah\)-admissible.

We can now prove Theorem~\ref{thmD}.
\begin{theorem}\label{cor:CP-scalar}
Let \(m\) be odd. If \(g\) is a Riemannian metric on \(\CP^m\) with \(\Scal_g\ge4m(m+1)\), then
\[
\stsys_2(\CP^m,g)
\le
\pi.
\]
\end{theorem}

\begin{proof}
Apply Theorem \ref{cor:theta-scalar-general} with \(\sigma=4m(m+1)\) and mod \(N=m+1\) and the mod $N$ class $ah$. Since \(b_2(\CP^m)=1\), this gives
\[\stsys_2(\CP^m,g) \leq  \frac{2\pi}{m+1} \Theta_{1,m+1} .\]
By Proposition \ref{prop:theta-12}, we have \(\Theta_{1,m+1}=(m+1)/2\). Hence \(\stsys_2(\CP^m,g)\le\pi\).
\end{proof}
\bibliographystyle{alpha}
\bibliography{bib}

@article{Banaszczyk1,
    AUTHOR = {Banaszczyk, W.},
     TITLE = {Inequalities for convex bodies and polar reciprocal lattices
              in {${\bf R}^n$}},
   JOURNAL = {Discrete Comput. Geom.},
  FJOURNAL = {Discrete \& Computational Geometry. An International Journal
              of Mathematics and Computer Science},
    VOLUME = {13},
      YEAR = {1995},
    NUMBER = {2},
     PAGES = {217--231},
      ISSN = {0179-5376,1432-0444},
   MRCLASS = {11H60 (11H06 52C07)},
  MRNUMBER = {1314964},
MRREVIEWER = {Martin\ Henk},
       DOI = {10.1007/BF02574039},
       URL = {https://doi.org/10.1007/BF02574039},
}

@article {Banaszczyk2,
    AUTHOR = {Banaszczyk, W.},
     TITLE = {Inequalities for convex bodies and polar reciprocal lattices
              in {$\bold R^n$}. {II}. {A}pplication of {$K$}-convexity},
   JOURNAL = {Discrete Comput. Geom.},
  FJOURNAL = {Discrete \& Computational Geometry. An International Journal
              of Mathematics and Computer Science},
    VOLUME = {16},
      YEAR = {1996},
    NUMBER = {3},
     PAGES = {305--311},
      ISSN = {0179-5376,1432-0444},
   MRCLASS = {11H60 (11H06 52C07)},
  MRNUMBER = {1410163},
MRREVIEWER = {Martin\ Henk},
       DOI = {10.1007/BF02711514},
       URL = {https://doi.org/10.1007/BF02711514},
}

@article{Federer,
    AUTHOR = {Federer, Herbert},
     TITLE = {Real flat chains, cochains and variational problems},
   JOURNAL = {Indiana Univ. Math. J.},
  FJOURNAL = {Indiana University Mathematics Journal},
    VOLUME = {24},
      YEAR = {1974/75},
     PAGES = {351--407},
      ISSN = {0022-2518,1943-5258},
   MRCLASS = {49F22},
  MRNUMBER = {348598},
MRREVIEWER = {F.\ J.\ Almgren, Jr.},
       DOI = {10.1512/iumj.1974.24.24031},
       URL = {https://doi.org/10.1512/iumj.1974.24.24031},
}

@incollection{Gromov,
    AUTHOR = {Gromov, M.},
     TITLE = {Positive curvature, macroscopic dimension, spectral gaps and
              higher signatures},
 BOOKTITLE = {Functional analysis on the eve of the 21st century, {V}ol.
              {II} ({N}ew {B}runswick, {NJ}, 1993)},
    SERIES = {Progr. Math.},
    VOLUME = {132},
     PAGES = {1--213},
 PUBLISHER = {Birkh\"{a}user Boston, Boston, MA},
      YEAR = {1996},
      ISBN = {0-8176-3855-5},
   MRCLASS = {53C21 (53C20 57R20)},
  MRNUMBER = {1389019},
MRREVIEWER = {Christopher\ W.\ Stark},
       DOI = {10.1007/s10107-010-0354-x},
       URL = {https://doi.org/10.1007/s10107-010-0354-x},
}

@article{GHK,
    AUTHOR = {Goodwillie, Thomas G. and Hebda, James J. and Katz, Mikhail
              G.},
     TITLE = {Extending {G}romov's optimal systolic inequality},
   JOURNAL = {J. Geom.},
  FJOURNAL = {Journal of Geometry},
    VOLUME = {114},
      YEAR = {2023},
    NUMBER = {2},
     PAGES = {Paper No. 23, 9},
      ISSN = {0047-2468,1420-8997},
   MRCLASS = {55N45 (55S30)},
  MRNUMBER = {4624541},
MRREVIEWER = {Yuli\ B.\ Rudyak},
       DOI = {10.1007/s00022-023-00685-3},
       URL = {https://doi.org/10.1007/s00022-023-00685-3},
}

@misc{Stryker,
  author        = {Stryker, D.},
  title         = {Stable 2-systole bounds in positive scalar curvature},
  year          = {2026},
  note          = {arXiv:2604.22106}
}

@misc{CHZ,
  author        = {Cecchini, S. and Hirsch, S. and Zeidler, R.},
  title         = {Stable 2-systoles, scalar curvature and spin-c comass bounds},
  year          = {2026},
  note          = {arXiv:2604.25900}
}

@incollection {four,
    AUTHOR = {Gromov, M.},
     TITLE = {Four lectures on scalar curvature},
 BOOKTITLE = {Perspectives in scalar curvature. {V}ol. 1},
     PAGES = {1--514},
 PUBLISHER = {World Sci. Publ., Hackensack, NJ},
      YEAR = {[2023] \copyright 2023},
      ISBN = {978-981-124-998-3; 978-981-124-935-8; 978-981-124-936-5},
   MRCLASS = {53C23 (53-02 53C21)},
  MRNUMBER = {4577903},
MRREVIEWER = {Mikhail\ G.\ Katz},
       DOI = {10.1142/9789},
       URL = {https://doi.org/10.1142/9789},
}

@article {KL88,
    AUTHOR = {Kannan, Ravi and Lov\'{a}sz, L\'{a}szl\'{o}},
     TITLE = {Covering minima and lattice-point-free convex bodies},
   JOURNAL = {Ann. of Math. (2)},
  FJOURNAL = {Annals of Mathematics. Second Series},
    VOLUME = {128},
      YEAR = {1988},
    NUMBER = {3},
     PAGES = {577--602},
      ISSN = {0003-486X,1939-8980},
   MRCLASS = {52A43 (11H06 11H31 52A45 90C10)},
  MRNUMBER = {970611},
MRREVIEWER = {J.\ M.\ Wills},
       DOI = {10.2307/1971436},
       URL = {https://doi.org/10.2307/1971436},
}

@article {aw12,
    AUTHOR = {Averkov, Gennadiy and Wagner, Christian},
     TITLE = {Inequalities for the lattice width of lattice-free convex sets
              in the plane},
   JOURNAL = {Beitr. Algebra Geom.},
  FJOURNAL = {Beitr\"{a}ge zur Algebra und Geometrie. Contributions to
              Algebra and Geometry},
    VOLUME = {53},
      YEAR = {2012},
    NUMBER = {1},
     PAGES = {1--23},
      ISSN = {0138-4821,2191-0383},
   MRCLASS = {52C05 (52A38 52A40 52C15)},
  MRNUMBER = {2890359},
MRREVIEWER = {Mar\'{\i}a\ A.\ Hern\'{a}ndez Cifre},
       DOI = {10.1007/s13366-011-0028-8},
       URL = {https://doi.org/10.1007/s13366-011-0028-8},
}

@article {hebda,
    AUTHOR = {Hebda, James J.},
     TITLE = {The collars of a {R}iemannian manifold and stable isosystolic
              inequalities},
   JOURNAL = {Pacific J. Math.},
  FJOURNAL = {Pacific Journal of Mathematics},
    VOLUME = {121},
      YEAR = {1986},
    NUMBER = {2},
     PAGES = {339--356},
      ISSN = {0030-8730,1945-5844},
   MRCLASS = {53C20},
  MRNUMBER = {819193},
MRREVIEWER = {Karsten\ Grove},
       URL = {http://projecteuclid.org/euclid.pjm/1102702435},
}

@incollection {gromov-systolic,
    AUTHOR = {Gromov, Mikhael},
     TITLE = {Systoles and intersystolic inequalities},
 BOOKTITLE = {Actes de la {T}able {R}onde de {G}\'{e}om\'{e}trie
              {D}iff\'{e}rentielle ({L}uminy, 1992)},
    SERIES = {S\'{e}min. Congr.},
    VOLUME = {1},
     PAGES = {291--362},
 PUBLISHER = {Soc. Math. France, Paris},
      YEAR = {1996},
      ISBN = {2-85629-047-7},
   MRCLASS = {53C22 (53-02 53C20 53C23)},
  MRNUMBER = {1427763},
MRREVIEWER = {Andrea\ Sambusetti},
}

@article {babenko-katz,
    AUTHOR = {Babenko, Ivan and Katz, Mikhail},
     TITLE = {Systolic freedom of orientable manifolds},
   JOURNAL = {Ann. Sci. \'{E}cole Norm. Sup. (4)},
  FJOURNAL = {Annales Scientifiques de l'\'{E}cole Normale Sup\'{e}rieure.
              Quatri\`eme S\'{e}rie},
    VOLUME = {31},
      YEAR = {1998},
    NUMBER = {6},
     PAGES = {787--809},
      ISSN = {0012-9593},
   MRCLASS = {53C23},
  MRNUMBER = {1664222},
MRREVIEWER = {Andrea\ Sambusetti},
       DOI = {10.1016/S0012-9593(99)80003-2},
       URL = {https://doi.org/10.1016/S0012-9593(99)80003-2},
}

@incollection {Katz-suciu,
    AUTHOR = {Katz, Mikhail G. and Suciu, Alexander I.},
     TITLE = {Volume of {R}iemannian manifolds, geometric inequalities, and
              homotopy theory},
 BOOKTITLE = {Tel {A}viv {T}opology {C}onference: {R}othenberg {F}estschrift
              (1998)},
    SERIES = {Contemp. Math.},
    VOLUME = {231},
     PAGES = {113--136},
 PUBLISHER = {Amer. Math. Soc., Providence, RI},
      YEAR = {1999},
      ISBN = {0-8218-1362-5},
   MRCLASS = {53C23 (55Q15)},
  MRNUMBER = {1705579},
MRREVIEWER = {Andrea\ Sambusetti},
       DOI = {10.1090/conm/231/03357},
       URL = {https://doi.org/10.1090/conm/231/03357},
}

@inproceedings {freedman,
    AUTHOR = {Freedman, Michael H.},
     TITLE = {{$Z_2$}-systolic-freedom},
 BOOKTITLE = {Proceedings of the {K}irbyfest ({B}erkeley, {CA}, 1998)},
    SERIES = {Geom. Topol. Monogr.},
    VOLUME = {2},
     PAGES = {113--123},
 PUBLISHER = {Geom. Topol. Publ., Coventry},
      YEAR = {1999},
   MRCLASS = {53C23 (53C20 81P68)},
  MRNUMBER = {1734404},
MRREVIEWER = {Athanase\ Papadopoulos},
       DOI = {10.2140/gtm.1999.2.113},
       URL = {https://doi.org/10.2140/gtm.1999.2.113},
}

@article {bangert-katz,
    AUTHOR = {Bangert, Victor and Katz, Mikhail},
     TITLE = {Stable systolic inequalities and cohomology products},
      NOTE = {Dedicated to the memory of J\"{u}rgen K. Moser},
   JOURNAL = {Comm. Pure Appl. Math.},
  FJOURNAL = {Communications on Pure and Applied Mathematics},
    VOLUME = {56},
      YEAR = {2003},
    NUMBER = {7},
     PAGES = {979--997},
      ISSN = {0010-3640,1097-0312},
   MRCLASS = {53C23 (53C20 53C65)},
  MRNUMBER = {1990484},
MRREVIEWER = {Andrea\ Sambusetti},
       DOI = {10.1002/cpa.10082},
       URL = {https://doi.org/10.1002/cpa.10082},
}

@book {gromov-metric,
    AUTHOR = {Gromov, Misha},
     TITLE = {Metric structures for {R}iemannian and non-{R}iemannian
              spaces},
    SERIES = {Progress in Mathematics},
    VOLUME = {152},
      NOTE = {Based on the 1981 French original [MR0682063 (85e:53051)],
              With appendices by M. Katz, P. Pansu and S. Semmes,
              Translated from the French by Sean Michael Bates},
 PUBLISHER = {Birkh\"{a}user Boston, Inc., Boston, MA},
      YEAR = {1999},
     PAGES = {xx+585},
      ISBN = {0-8176-3898-9},
   MRCLASS = {53C23 (53-02)},
  MRNUMBER = {1699320},
MRREVIEWER = {Igor\ Belegradek},
}

@article{Sha2026,
  author = {Zehao Sha},
  title = {The 2-systole on compact K{\"a}hler surfaces with positive scalar curvature},
  journal = {arXiv preprint},
  year = {2026},
  eprint = {2601.02901},
  archivePrefix = {arXiv},
  primaryClass = {math.DG}
}

@article{Orikasa2025,
  author = {Shunichiro Orikasa},
  title = {Systolic Inequality and Scalar Curvature},
  journal = {arXiv preprint},
  year = {2025},
  eprint = {2509.17376},
  archivePrefix = {arXiv},
  primaryClass = {math.DG}
}

@article {richard,
    AUTHOR = {Richard, Thomas},
     TITLE = {On the 2-systole of stretched enough positive scalar curvature
              metrics on {$\Bbb S^2\times\Bbb S^2$}},
   JOURNAL = {SIGMA Symmetry Integrability Geom. Methods Appl.},
  FJOURNAL = {SIGMA. Symmetry, Integrability and Geometry. Methods and
              Applications},
    VOLUME = {16},
      YEAR = {2020},
     PAGES = {Paper No. 136, 7},
      ISSN = {1815-0659},
   MRCLASS = {53C42 (53C20)},
  MRNUMBER = {4188852},
MRREVIEWER = {Mikhail\ G.\ Katz},
       DOI = {10.3842/SIGMA.2020.136},
       URL = {https://doi.org/10.3842/SIGMA.2020.136},
}

@article {zhu,
    AUTHOR = {Zhu, Jintian},
     TITLE = {Rigidity of area-minimizing {$2$}-spheres in {$n$}-manifolds
              with positive scalar curvature},
   JOURNAL = {Proc. Amer. Math. Soc.},
  FJOURNAL = {Proceedings of the American Mathematical Society},
    VOLUME = {148},
      YEAR = {2020},
    NUMBER = {8},
     PAGES = {3479--3489},
      ISSN = {0002-9939,1088-6826},
   MRCLASS = {53C24 (53C42)},
  MRNUMBER = {4108854},
MRREVIEWER = {Otis\ Chodosh},
       DOI = {10.1090/proc/15033},
       URL = {https://doi.org/10.1090/proc/15033},
}

@article {xu,
    AUTHOR = {Xu, Kai},
     TITLE = {A topological gap theorem for the {$\pi_2$}-systole of
              positive scalar curvature 3-manifolds},
   JOURNAL = {Duke Math. J.},
  FJOURNAL = {Duke Mathematical Journal},
    VOLUME = {174},
      YEAR = {2025},
    NUMBER = {8},
     PAGES = {1647--1664},
      ISSN = {0012-7094,1547-7398},
   MRCLASS = {53C20 (53C21 53C23 53E10)},
  MRNUMBER = {4916112},
MRREVIEWER = {Chady\ El Mir},
       DOI = {10.1215/00127094-2024-0061},
       URL = {https://doi.org/10.1215/00127094-2024-0061},
}

@article{Kumar2026,
  author = {A. Kumar},
  title = {An improved cowaist inequality for line bundles and consequences},
  journal = {arXiv preprint},
  year = {2026},
  eprint = {2604.26891},
  archivePrefix = {arXiv},
  primaryClass = {math.DG}
}

\end{document}